\documentclass[reqno,11pt]{amsproc}
\usepackage[utf8]{inputenc}
\usepackage[english]{babel}
\usepackage{amssymb,amsmath}
\usepackage{indentfirst}
\usepackage[usenames]{color}
\usepackage{graphicx}
\usepackage[left=2cm,right=2cm,top=1.5cm,bottom=1.5cm,bindingoffset=0cm]{geometry}
\usepackage{amsthm}
\usepackage{amsfonts}
\usepackage{graphicx}
\graphicspath{{pictures/}}
\DeclareGraphicsExtensions{.pdf,.png,.jpg}
\usepackage{wrapfig}

\theoremstyle{plain}  
\newtheorem*{acknowledgements}{Acknowledgements}
\newtheorem{theorem}{\bf Theorem}

\newtheorem{lemma}{\bf Lemma}
\newtheorem{corollary}{\bf Corollary}

\newtheorem{oldthm}{\bf Theorem}

\newtheorem{oldconj}{\bf Conjecture}

\newtheorem{remark}{\bf Remark}

\renewcommand{\ll}{\lambda}

\newcommand{\de}{\operatorname{def}}

\usepackage{xcolor}
\usepackage{hyperref}

\usepackage{float}

\hypersetup{pdfstartview=FitH, linkcolor=linkcolor,urlcolor=urlcolor, colorlinks=true, linkcolor=blue, citecolor=blue, urlcolor=blue}

\newcommand\Colorhref[3][blue]{\href{#2}{\small\color{#1}#3}}

\begin{document}

\title{Upper and lower estimates for integer complexity}
\author{Sergei Konyagin}
\address{Sergei Konyagin, Steklov Mathematical Institute, Gubkina Str. 8, Moscow 119991, Russia}
\email{konyagin@mi-ras.ru}

\author{Kristina Oganesyan}
\address{Kristina Oganesyan, Steklov Mathematical Institute, Gubkina Str. 8, Moscow 119991, Russia}
\email{kristina.oganesyan@mi-ras.ru}
\date{}

\begin{abstract}
Let $\|n\|$ stand for the integer complexity of the number $n$, i.e. for the least number of $1$'s needed to write $n$ using arbitrary many additions, multiplications, and parentheses. The two-sided inequality $3\log_3 n\leq\|n\|\leq 3\log_2 n$ for all $n$ is well known and reveals the logarithmic behaviour of the complexity function $\|n\|$. While the lower bound $3\log_3 n$ is attained infinitely many times at powers of $3$, the best upper estimate is still unknown, although there are some improvements of the trivial bound $3\log_2 n$. Besides, for $``$typical$"$ numbers, i.e. for almost all numbers $n$, the better inequality $\|n\|\leq C_{avg}\log n$ holds, where, importantly, $C_{avg}\approx { 3.236}<\sup_{n} \frac{\|n\|}{\log n}$ .

We show that in fact $\|n\|\leq C_{avg}\log n+o(\log n)$ as $n\to\infty$, which, in particular, yields that $\limsup\limits_{n\to\infty}\frac{\|n\|}{\log n}\leq C_{avg}$. We also obtain the first nontrivial lower bound $\|n\|\geq 3.06\log_3 n$ for almost all numbers $n$. 
\end{abstract}



\maketitle

\section{Introduction}

Define the integer complexity (or just complexity) $\|n\|$ of a positive integer $n$ to be the least number of $1$'s needed to write $n$ using arbitrary many additions $+$, multiplications $\cdot$, and parentheses $(\;)$. For example, one can check that if $n\leq 5$, then there is no way to write $n$ using less than $n$ copies of $1$, which means $\|n\|=n$ for these numbers, while $6$ can be already represented as $(1+1)(1+1+1)$ and has therefore the complexity $\|6\|=5$. Importantly enough, the function of integer complexity is not monotone: for instance, an optimal expression for $11$ is 
 $(1+1+1)(1+1+1)+1+1$, whence $\|11\|=8$, however,  $\|12\|=7$ due to the factorization $12=(1+1)(1+1)(1+1+1)$. 
 
 Introduced by K. Mahler and J. Popken \cite{MP} back in 1953 and divulged by R.K. Guy \cite{G1}, integer complexity became an extensively studied subject (see \cite{G1,R,Ar} in the first place), both analytically and experimentally.

 It is easy to see that once having expressed two numbers $a$ and $b$ using $1$'s, additions, multiplications, and parentheses, one can straightforwardly derive a corresponding expression for $a+b$ and $ab$ with the number of $1$'s involved being equal to the sum of those for $a$ and $b$. Moreover, for a given number $n\geq 2$, one can consider an optimal expression and recover the order of the operations $+$ and $\cdot$ in it dictated by the parentheses, so that the last operation will determine a pair of numbers $a$ and $b$ with $\|a\|+\|b\|=\|n\|$ and satisfying either $ab=n$ or $a+b=n$. Thus, for $n\geq 2$,
 \begin{align*}
     \|n\|=\min_{\underset{(a+b-n)(ab-n)=0}{a,b<n}}(\|a\|+\|b\|),
 \end{align*}
 which provides a recurrence relation for computing $\|n\|$. We refer the reader to \cite{AL,CEHMPSSV,IBCOOP,H} and references therein to get acquainted with the known algorithms for calculating complexities of positive integers. It is worth noting that there is an available online calculator by J. Iraids (see \cite{Sl}) that allows one to compute $\|n\|$ for all $n\leq 10^{12}$. 

The behaviour of $\|n\|$  is well known to be logarithmic, namely, the following upper and lower bounds, attributed to J. Selfridge and D. Coppersmith (see \cite{G1}), hold: 
 \begin{align*}
     3\log_3 n\leq \|n\|\leq 3\log_2 n.
 \end{align*}
Here the inequality $\|n\|\leq 3\log_2 n$ follows immediately from the binary expansion of the number $n$, while the other one needs a more delicate argument (see e.g. \cite[Cor. 3]{Ar}). One can see that the lower bound is attained infinitely many times at the numbers of the form $n=3^k$, whence, in particular,\footnote{Here and below, $\log n=\log_e n$.}
\begin{align}\label{liminf}
    \liminf_{n\to\infty}\frac{\|n\|}{\log n}=\frac{3}{\log 3}\approx 2.731.
\end{align}
For the upper bound, an improvement was obtained in \cite{Z}, however, obtaining the sharp bound remains an open problem. Up to now, judging by the numerical evidences, it looks like the most suitable candidate for being $``$the most complex$"$ number maximizing the ratio $\|n\|/\log n$ is the number $1439$ with $\|1439\|=26$ and $\|1439\|/\log 1439\approx 3.575$ (compare this with the trivial upper bound $3/\log 2\approx 4.328$).

A portion of interesting questions arises if we do not aim to find the worst number in terms of the complexity/logarithm ratio but rather intent to understand the $``$typical$"$ behaviour of this ratio. In this regard, a classical approach (see e.g. \cite[Sec. 4]{Ar}) of averaging binary digits gives for almost all $n$ (i.e. for all numbers $n\in\mathbb N$ but a set of zero density),
\begin{align*}
    \frac{\|n\|}{\log n}\leq \frac{5}{2\log 2}\approx 3.607.
\end{align*}
Roughly speaking, this inequality follows from the fact that the number $k(n)$ of $1$'s in the binary representation of $n$ is in average close to $n/2$, whence it remains to observe that $\|n\|\leq 2\log_2 n+k(n)\log_2 n$. As reported in \cite{G1}, J. Isbell, using base $24$, improved this bound to  $265/(24\log 24)\approx 3.474$.  
 After that, with the help of powerful calculations in bases $2^9 3^8$, $2^{11}3^9$, {and $2^{15}3^{14}5$} new bounds were obtained in \cite{AL}, {\cite{CEHMPSSV}, and \cite{Am}}, so that the best bound up to date is
\begin{align}\label{avg}
    \frac{\|n\|}{\log n}\leq C_{avg}:={\frac{69451528190836
}{2^{15}3^{14}5\log (2^{15}3^{14}5)}\approx 3.236}\quad\text{for almost all}\;n.
\end{align}
{It is worth noting that based on the ideas from \cite{S}, a better upper bound $\approx 3.204$ for almost all numbers in the sense of logarithmic density was obtained in \cite{Sh} 
 and eventually refined to $\approx 3.16$ in \cite{Am}.}

In this paper, invoking the exponential sum technique, we show that estimate \eqref{avg} holds for all numbers $n$ up to an error term that vanishes as $n$ tends to infinity. Namely, we prove 
\begin{theorem}\label{th11}
For any positive integer $n\geq 3$,
\begin{align*}
\|n\|\leq C_{avg}\log n+C(\log n)^{2/3}(\log\log n)^{4/3},
\end{align*}
where $C$ is an absolute constant.
\end{theorem}
In fact, we will provide a more general statement (see Theorem \ref{th1} in Section \ref{Upper}) that enables one to derive general upper bounds for integer complexity using averaging of digits in a given base. 

In particular, Theorem \ref{th11} readily implies

\begin{align}\label{cor1}
\limsup_{n\to\infty}\frac{\|n\|}{\log n}\leq C_{avg}.
\end{align}

Although there was no nontrivial lower bound for $\limsup\limits_{n\to\infty}\frac{\|n\|}{\log n}$ up to now, there were some evidences suggesting one. To this end, the following currently open question was initially posed by J. Selfridge (see \cite{G1}): does there exist any positive integer $a$ with $\|2^a\|<2a$? It was shown in \cite{IBCOOP}, that $\|2^a\|=2a$ as long as $2^a<10^{12}$, and that moreover, $\|2^a 3^b 5^c\|=2a+3b+5c$ for all $a+b+c>0$ and $c<6$ provided that $2^a 3^b 5^c<10^{12}$ as well. This heavily supports Conjecture \ref{conjec} \cite[Hyp. 5]{IBCOOP} below (which is a generalization of the one in \cite{G1} corresponding to the case $c=0$) and suggests a negative answer to Selfridge's question, which in turn would imply the nontrivial lower bound $\limsup\limits_{n\to\infty}\frac{\|n\|}{\log n}\geq 2/\log 2.$
\begin{oldconj}\label{conjec} For any nonnegative integers $a,b,c$ with $a+b+c>0$ and $c<6$,
\begin{align*}
    \|2^a 3^b 5^c\|=2a+3b+5c.
\end{align*}
\end{oldconj}
According to {\cite{A2}}, 
Conjecture \ref{conjec} is true provided that $a\leq {48}$ and $c=0$.

As for a nontrivial lower bound for almost all numbers, nothing was known so far. Our next result provides the first breakaway from the trivial bound.

\begin{theorem}\label{th_low}
    For almost all numbers $n$,
\begin{align*}
    \|n\|> 3.06\log_3n. 
\end{align*}
\end{theorem}
In particular, combining this result with \eqref{cor1}, we can formulate 
\begin{corollary} We have
    \begin{align*}
        2.785\leq \limsup_{n\to\infty}\frac{\|n\|}{\log n}\leq {3.237}.
    \end{align*}
\end{corollary}
As a final remark, out of pure curiosity and for the reader's convenience, we compare the three values
\begin{align*}
    \frac{3}{\log 3}\approx 2.731<\frac{3.06}{\log 3}\approx 2.785<\frac{2}{\log 2}\approx 2.885
\end{align*}
that correspond to the three mentioned lower bounds for $\|n\|/\log n$: the trivial one, the one provided by Theorem \ref{th_low} for almost all numbers, and the one that would follow from the particular case $b=c=0$ of Conjecture \ref{conjec} for almost all $n$. 

\section{Proof of an upper bound for integer complexity}\label{Upper}
Following {\cite{AL}}, denote
\begin{align*}
D(b,r):=\sup_{n\in\mathbb N}(\|bn+r\|-\|n\|).
\end{align*} 
Importantly, {the known proofs of inequalities of the form $\|n\|\leq C\log n$ for almost all $n$ are based on the fact that 
\begin{align*}
    C\leq \frac{1}{b\log b}\sum_{r=0}^{b-1}D(b,r)
\end{align*}
for any $b\geq 2$ (see \cite[Prop. 12]{AL}). In particular,} \eqref{avg} {in \cite{Am}}
 follows from the inequality
\begin{align}\label{eq_ceh}
\frac{1}{b_0\log b_0}\sum_{r=0}^{b_0-1}D(b_0,r)\leq C_{avg},
\end{align}
where {$b_0=2^{15}3^{14}5$}. 

We are going to show that by means of a trick involving exponential sums, one can perform some kind of averaging of digits not over an integer interval (as it was done for \eqref{avg}) but dealing with a single integer. 

\begin{theorem}\label{th1}
For any positive integers $m>1$ and $n\geq 3$,
\begin{align*}
\|n\|\leq \frac{1}{m\log m}\sum_{r=0}^{m-1}D(m,r)\log n+C(\log n)^{2/3}(\log\log n)^{4/3}+C\log m,
\end{align*}
where $C$ is an absolute constant.
\end{theorem}

\begin{proof}
Note that for $m>n$ the statement is obvious, therefore, from now on we assume that $m\leq n$. For a fixed $m$, let the integer parameters $K=K(n,m)=o_m(n)$ and $1<p=p(n,m)=o_m(n)$, which will be chosen later, tend to infinity along with $n$. Fix some positive integer $j$ satisfying 
\begin{align}\label{cond11}
\sqrt{2}m^{j}K^5\leq n
\end{align}
and
\begin{align}\label{cond22}
m^j>p^3,
\end{align}
(it could happen that there is no such $j$, then we do not need the corresponding estimates). For each $k=K,...,2K-1$, denote by $r_k$ the remainder of $n$ modulo $k$ and consider the set of points $S_j:=\{(n-r_k)m^{-j}k^{-1}\}_{k=K}^{2K-1}$. Due to the Erd\"os-Tur\'an theorem \cite[\S 2, Th. 2.5]{KN}, we have the following estimate for the discrepancy:\footnote{Here and further, for two nonnegative functions $f$ and $g$, we write $f\lesssim g$ (equivalently, $g\gtrsim f$) meanig that there exists a constant $C>0$ such that $f\leq Cg$, while the notation $f\asymp g$ stands for $f\lesssim g\lesssim f$.}

\begin{align}\label{discr}
\Delta_{K,j}:&=\sup_{[\alpha,\beta]\in[0,1]}\Big|\frac{|\{x\in S_j:\{x\}\in[\alpha,\beta]\}|}{K}-(\beta-\alpha)\Big|\nonumber\\
&\lesssim \frac{1}{p}+\sum_{h=1}^p\Big(\frac{1}{h}-\frac{1}{p+1}\Big)\Big|\frac{1}{K}{\sum_{k=K}^{2K-1}}e\Big(\frac{h(n-r_k)}{m^{j}k}\Big)\Big|\nonumber\\
&\lesssim \frac{1}{p}+\sum_{h=1}^p\Big(\frac{1}{h}-\frac{1}{p+1}\Big)\Big|\frac{1}{K}{\sum_{k=K}^{2K-1}}e\Big(\frac{hn}{m^{j}k}\Big)\Big|,
\end{align}
where we used the inequality $|e(h(n-r_k)m^{-j}k^{-1})-e(hnm^{-j}k^{-1})|\lesssim hr_km^{-j}k^{-1}\leq p^{-2}$, valid in light of \eqref{cond22}.

For the sake of completeness, we provide here a statement \cite[Th. 8.25]{IK} that we will use further:

\begin{oldthm}\label{thm8.25}
    Let $f(x)$ be a smooth function on $[N,2N]$ such that for all $x$ {and all $j\geq 1$},
    \begin{align*}
        \alpha^{-j^3}F\leq \frac{x^j}{j!}|f^{(j)}(x)|\leq \alpha^{j^3}F,
    \end{align*}
    where $F\geq N^4$ and $\alpha\geq 1$. Then
    \begin{align*}
        \Big|\sum_{a<n<b}e(f(n))\Big|\lesssim \alpha N \exp(-2^{-18}(\log N)^3(\log F)^{{-2}}).
    \end{align*}
\end{oldthm}
For a fixed $h$, $1\leq h\leq p,$ let us estimate the inner exponential sum in \eqref{discr} according to Theorem \ref{thm8.25} assuming $f(x):=nhm^{-j}/x,\;N:=K,\;\alpha:=\sqrt{2},\;F:=hnm^{-j}2^{-1/2}/K$. Notice that $F\geq K^4$ is fulfilled due to \eqref{cond11}, and by the definition of $\alpha$ and $F$, 
\begin{align*}
\frac{F}{\sqrt 2}\leq \frac{x^{\ell}}{\ell!}|f^{(\ell)}(x)|=\frac{hn}{m^jx}\leq \sqrt 2F
\end{align*}
for all $\ell\geq 1$ and $K\leq x\leq 2K$. Hence, we derive from \eqref{discr} and Theorem \ref{thm8.25} that
\begin{align}\label{Delta}
\Delta_{K,j}&\lesssim \frac{1}{p}+\sum_{h=1}^p\Big(\frac{1}{h}-\frac{1}{p+1}\Big)\exp\Big(-2^{-18}\frac{(\log K)^3}{(\log\frac{hn}{\sqrt{2}m^{j}K})^2}\Big)\nonumber\\
&\lesssim \frac{1}{p}+\log p\exp\Big(-2^{-18}\frac{(\log K)^3}{(\log\frac{pn}{\sqrt{2}m^{j}K})^2}\Big)\nonumber\\
&\lesssim \frac{1}{p}+\log p\exp\Big(-2^{-18}\frac{(\log K)^3}{(\log\frac{n}{p^2K})^2}\Big),
\end{align}
where the implied constant does not depend on $m$. Thus, denoting by $[\ell]_i$ the $i$'s digit of $\ell$ in base $m$, we obtain for $j,\;\log_m p^3<j<\log_m\frac{n}{2K^5}$,
\begin{align*}
    \sum_{k=K}^{2K-1}D(m,[(n-r_k)/k]_j)&=\sum_{r=0}^{m-1} \sum_{k=K}^{2K-1}\chi_{\{[(n-r_k)/k]_j=r\}}D(m,r)\\
    &\leq K\sum_{r=0}^{m-1}\Big(\frac{D(m,r)}{m}+\Delta_{K,j}\max_{0\leq l<m}D(m,l)\Big),
\end{align*}
whence (in light of conditions \eqref{cond11} and \eqref{cond22}) due to \eqref{Delta},
\begin{align*}
&\sum_{\log_m p^3<j<\log_m\frac{n}{2K^5}}\sum_{k=K}^{2K-1}D(m,[(n-r_k)/k]_j)<K\log_m n\frac{1}{m}\sum_{r=0}^{m-1}D(m,r)\\
&\quad+K\cdot\mathcal{O}\Big(\frac{\log_mn}{p}+\log_mn\log p\exp\Big(-2^{-18}\frac{(\log K)^3}{(\log\frac{n}{pK})^2}\Big)\Big)\max_{0\leq r<m}D(m,r).
\end{align*}
Hence, noting that $D(m,r)\lesssim \log m$ for $r<m$, we get for some $k_0\in[K,2K-1]$,
\begin{align*}
\sum_{\log_m p^3<j<\log_m\frac{n}{2K^5}}&D(m,[(n-r_{k_0})/{k_0}]_j)<\log_m n\frac{1}{m}\sum_{r=0}^{m-1}D(m,r)\\
&+\mathcal{O}\Big(\frac{\log n}{p}+\log n\log p\exp\Big(-2^{-18}\frac{(\log K)^3}{(\log\frac{n}{pK})^2}\Big)\Big),
\end{align*}
so that using the estimate
\begin{align*}
&\Big(\sum_{j=0}^{\lfloor\log_m p^3\rfloor}+\sum_{j=\lceil\log_m\frac{n}{2K^5}\rceil}^{\lceil \log_m n\rceil}\Big)D(m,[(n-r_{k_0})/{k_0}]_j)=\mathcal{O}\Big(\lceil\log_m p^3\rceil+\lceil\log_m K^5\rceil\Big)\max_{0\leq r<m}D(m,r),
\end{align*}
we conclude
\begin{align*}
\sum_{j=0}^{\lceil\log_m n\rceil}&D(m,[(n-r_{k_0})/{k_0}]_j)<\log_m n\frac{1}{m}\sum_{r=0}^{m-1}D(m,r)\\
&+\mathcal{O}\Big(\frac{\log n}{p}+\log n\log p\exp\Big(-2^{-18}\frac{(\log K)^3}{(\log\frac{n}{pK})^2}\Big)+\log p^3+\log K^5+\log m\Big).
\end{align*}
Assuming 
$p:=\lfloor\log n/\log\log n\rfloor$ (according to our assumtion, $n\geq 3$) and $K:=\lfloor {\exp}((\log n)^{2/3}(W(3\cdot 2^{-18}\log n))^{1/3})\rfloor$, where $W(x)$ is the Lambert function (i.e. $x=W(x)e^{W(x)}$), we obtain $(\log n)/p\asymp\log p^3\asymp \log\log n$ and also $$\log n\exp\Big(-2^{-18}\frac{(\log K)^3}{(\log\frac{n}{p^2K})^2}\Big)\asymp\log n\exp\Big(-2^{-18}\frac{(\log K)^3}{(\log n)^2}\Big)\asymp \log K^5\asymp(\log n)^{2/3}(\log\log n)^{1/3}.$$ Thus,
\begin{align*}
\|n\|\leq \|n_{k_0}\|+\|k_0\|+\|r_{k_0}\|&\leq \frac{1}{m\log m}\sum_{r=0}^{m-1}D(m,r)\log n\\
&+\mathcal{O}\Big(\log \log n+\log\log n(\log n)^{2/3}(\log\log n)^{1/3}+\log m\Big)\\
&=\frac{1}{m\log m}\sum_{r=0}^{m-1}D(m,r)\log n+\mathcal O((\log n)^{2/3}(\log\log n)^{4/3}+\log m),
\end{align*}
which concludes the proof.
\end{proof}

\section{Proof of Theorem \ref{th_low}}\label{Lower}

Define $\de(n):=\|n\|-3\log_3 n$ to be the defect of the number $n$. This notion, in many occasions more reasonable to deal with rather than the notion of integer complexity itself, was introduced in \cite{AZ} and gave rise to a classification theorem {\cite[Th. 29]{AZ}} that suggested a way to inductively estimate the quantity of numbers of small defect in a given interval. In particular, the classification theorem implied {\cite[Th. 52]{AZ}} that $|\{n\leq x:\;\de(n)<r\}|\asymp_r (\log x)^{\lfloor r\rfloor +1}$. We note that it was also shown in \cite[Th. 1.5]{A}, that for any value $\delta>0$, there exists a finite set $\mathcal T_{\delta}$ of multilinear polynomials such that any number with the defect smaller than $\delta$ can be represented as the value of a polynomial in $\mathcal T_{\delta}$ at a tuple of nonnegative powers of $3$.

We are going to use the classification theorem {\cite[Th. 29]{AZ}} in a more subtle way, so that it will provide a nontrivial lower bound for the complexity of almost all numbers.

\begin{proof}[Proof of Theorem \ref{th_low}.]
Let $B$ stand for the set of all leaders, i.e. numbers $n$ that are either not divisible by $3$ or satisfy the inequality $\|n\|<\|n/3\|+3$. For a fixed $\sigma\in (0,1)$, we denote 
$$D_{k,\sigma}:=\{n:\;\de(n)\in[(k-1)\sigma,k\sigma)\}$$
and
$$U_S^{\sigma}(k,m):=|S\cap D_{k,\sigma}\cap (3^{m-1},3^m]|,\quad U_S^{\sigma}(k):={\sum_{m\geq 1}}U_S^{\sigma}(k,m),$$ 
for any set $S\subset\mathbb N$ and any $k,m\in\mathbb N$.

We say that a positive integer $n$ is multiplicatively (additively) irreducible if for any pair of numbers $a$ and $b$ with $ab=n$ ($a+b=n$) we have $\|n\|<\|a\|+\|b\|$. Following \cite{AZ}, let $T_{\sigma}$ consist of $1$ and all multiplicatively irreducible numbers $n<(3^{(1-\sigma)/3}-1)^{-1}+1$ that do not satisfy $\|n\|=\|n-b\|+\|b\|$ for any additively irreducible $1<b\leq n/2$.

For the sake of convenience, in this section, we will usually omit the dependence on $\sigma$ of the quantities introduced above and assume from the beginning that 
\begin{align}\label{de(2)}
\sigma:=0.48<4.5\de(2)=4.5\log_3(3^22^{-3})\approx 0.482. 
\end{align} 

We are going to prove by induction on $k$ that for $k\geq 3,$
\begin{align}\label{we_want}
U_B(k,m)\leq \ll\frac{(Cm)^{k-2}}{k^{k+1}},
\end{align}
where $C$ and $\ll$ will be defined later.

Observe that a number $n\in T_{0.48}$ satisfies
\begin{align*}
    n\leq\lfloor(3^{\frac{1-0.48}{3}}-1)^{-1}+1\rfloor =5,
\end{align*}
whence, noting that $\|5\|=\|2\|+\|3\|$ and $\|4\|=\|2\|+\|2\|,$ we conclude that ${T_{0.48}=\{1,2,3,5\}}$. Besides, due to {\cite[Props. 36-44, Th. 31]{AZ}}, we have 
\begin{align*}
    B\cap D_1=\{3,2,4,8,16\}\quad\text{and}\quad B\cap D_2=\{32,5,64,7,10,128,14,20,256,28,40,19\},
\end{align*}
so that $U_B(1)=5$ and $U_B(2)=12$. Moreover, one can see that for any $m\geq 1$,
\begin{align}\label{U(2,m)}
    U_B(2,m)\leq 4.
\end{align}
Observe that if $U^{\sigma}_{\mathbb N}(k,m)\neq 0$, then {there exists a number $n\leq 3^m$ with the defect at least $(k-1)\sigma$, which in light of the inequality ${\|n\|\leq 3\log_2 n}$ yields}
\begin{align*}
    (k-1)\sigma{\leq \|n\|-3\log_3 n\leq 3\log_2 n-3\log_3 n}\leq 3m(\log_2 3-1),
\end{align*}
whence $k-1\leq 1.76m/\sigma=:\tau(\sigma)m$. So, in fact we aim to show that
\begin{align}\label{aim_to}
U_B(k,m)\leq u(k,m):=\ll\chi_{k\leq \tau m+1}\frac{(Cm)^{k-2}}{k^{k+1}}\quad\text{for}\quad k\geq 3,
\end{align}
where 
\begin{align*}
    \tau:=\tau(0.48)=\frac{1.76}{0.48}=\frac{11}{3}.
\end{align*}
From now on, we assume that $C$ fulfills the following conditions:
\begin{align}\label{00}
    C\geq \frac{{178}\cdot 81}{11} \ll^{-1},
\end{align}
\begin{align}\label{cond1}
    C\geq {2^{10}3^{-5}}\tau c^{-1},
\end{align}
\begin{align}\label{cond0}
    C\geq 127(2-3c)^{-1}e^{\eta}c_1^{-1},
\end{align}
\begin{align}\label{cond2}
    C\geq {38}e^{\eta}\ll c_2^{-1},
\end{align}
\begin{align}\label{cond4}
    C\geq {702}(2-3c)^{-1}c_3^{-1},
\end{align}
\begin{align}\label{cond5}
	C^2\geq {6536}e^{\eta}c(1-c)^{-1}c_4^{-1},  
\end{align}
\begin{align}\label{cond33}
    C\geq 11124e^{\eta} c^2(1-{1.3}c^3)^{-1}(2-3c)^{-1}c_5^{-1},
\end{align}
\begin{align}\label{cond8}
	C^2\geq {286749} \ll^{-1} c_6^{-1},    
\end{align}
where the coefficients $\eta>0$, $0<c<{0.9}$, and $c_j\in(0,1),\;1\leq j\leq 6,$ will be chosen later to satisfy $\sum_{j=1}^6 c_j=1$.

Note that for $C$ satisfying \eqref{cond1}, {provided that $u(k,m)\neq 0$}, we have 
\begin{align}\label{under_cond1}
    u(k-1,m)\leq cu(k,m)
\end{align}
for $k\geq 4$ and any $m$. Indeed, if $u(k,m)\neq 0$, then
\begin{align}\label{star}
    \frac{u(k,m)}{u(k-1,m)}=\frac{Cm}{k(\frac{k}{k-1})^{k}}\geq \frac{Cm}{k}\Big({\frac{4}{3}}\Big)^{-{4}}\geq C\frac{3}{4\tau}\cdot\Big({\frac{4}{3}}\Big)^{-{4}}\geq c^{-1}.
\end{align}

Before we start with the proof of \eqref{aim_to}, let us also put some restrictions on $m$. Namely, we are going to discard those values of $m$ for which the trivial estimate $U_B(k,m)\leq 3^m$ gives a better bound than that of \eqref{aim_to}. For this, we will assume that 
\begin{align}\label{2-600}
    \lambda\geq 2.5,\quad C\geq 780.
\end{align}
Then, in order to have $\ll(Cm)^{k-2}/k^{k+1}\leq 3^m$, for $k=3,4,5,6,7,8,9,$ it is necessary that $2.5(780m)^{k-2}/k^{k+1}\\\leq 3^m$, which yields $m\geq 5,12,19,26,33,41,48$, respectively. For $k\geq 10,$ we have $k^3\leq 2.5^{k-2}$, 
 whence as $C\geq 780$, it is necessary that $2.5(312m/k)^{m/\nu}\leq 3^m,$ where $\nu:=m/(k-2)$. Noting that for $k\geq 10,$ $\nu\leq 1.25m/k$, and checking therefore the condition $2.5(249.6\nu)^{m/\nu}\leq 3^m,$ we see that {there must hold $(249.6\nu)^{1/\nu}\leq 3,$ and thus either $\nu> 6.76$ or $\nu<0.005$. However, if $\nu<0.005$, then $m\leq 0.005(k-2)<(k-1)\tau,$ which is a contradiction, so that $\nu>6.76$.} Comparing this with the numerics for $k\leq 9$, we conclude that 
\begin{align}\label{eta}
    k-2\leq \frac{1}{6}m=:\eta m,\quad k\leq \frac{1}{3} m\qquad\text{for}\;k\geq 4.
\end{align}

Let us postpone the proof of the induction base for now and start with the induction step from $k-1$ to $k\geq 4$. 

According to {\cite[Th. 29]{AZ}}, for any $\sigma\in (0,1),$ any number $n\in B\cap D_k\setminus T_{\sigma},\;k\geq 3,$ can be efficiently represented in one of the following three forms:

(1) $n=uv$ with $u\in T_{\sigma},\;v\in B\cap D_1$.

(2) $n=uv$ with $u\in B\cap D_i,\;v\in B\cap D_j,\;i,j\leq k-1,\;i+j\leq k+1$.

(3) $n=(a+b)v$ with $v\in B\cap D_1\cup\{1\}$, $a\in \cup_{l\leq k-1}D_{l}$, and $b\leq a$ being an additively irreducible number with $\|b\|<k\sigma+3\log_32-\de(a)$.

For our choice \eqref{de(2)} of $\sigma=0.48$, let us estimate the number of representations $N_i=N_i(k,m),\;1\leq i\leq 3,$ of these three forms separately. First of all, we notice that $B\cap D_k\setminus T_{0.48}=B\cap D_k$.

\textbf{Case 1.} Observe that this case is impossible for $k\geq {4}$ and our choice of $\sigma=0.48$. Indeed,
\begin{align*}
    \{uv:\;u\in T_{0.48}\setminus\{1\},\;v\in B\cap D_1\}&=\{uv:\;u\in\{2,3,{5}\},\;v\in\{3,2,4,8,16\}\}\subset D_1\cup D_2\cup\{80\}.
\end{align*}
Hence, {as $80\in D_3$, we have} $N_1=0$ {whenever $k\geq 4$}.

\textbf{Case 2.}  Let $\widetilde{U}_B(p):=\max_l\sum_{s\leq p}U_B(s,l)$. For $l\geq 4$, we have $U_B(1,l)+U_B(2,l)=U_B(2,l)\leq 4$ by \eqref{U(2,m)}, while for $l<4$, one can check that $U_B(1,l)+U_B(2,l)\leq 5$, whence $\widetilde{U}_B(2)=5$. Also let $Q_{p}$ be the maximum over $l\in\mathbb N$ of the quantities of numbers in $D_p\cap(3^{l-1},3^l]$ that can be represented as $uv$ with $u,v\in B\cap(D_1\cup D_2)$. A direct computation, using the online calculator by J. Iraids (see \cite{Sl}), shows that $Q_p=0$ for $p\geq {5}$ and $Q_3=8,\;Q_4=9$. 
 
 {Note that the equality $n=uv$ with $u\in(3^{m_1-1},3^{m_1}]\cap D_{k_1}$ and $v\in(3^{m_2-1},3^{m_2}]\cap D_{k_2}$ implies $n\in (3^{m_1+m_2-2},3^{m_1+m_2}$, whence $m\leq m_1+m_2\leq m+1$, and also that $(k_1+k_2-2)\sigma\leq \de(n)=\de(u)+\de(v)<k_1\sigma+k_2\sigma$, so that $k\leq k_1+k_2\leq k+1$.} Thus, we obtain the bound
\begin{align*}
N_2&\leq \sum_{\underset{1\leq k_2\leq k_1\leq k-1}{{k\leq}k_1+k_2\leq k+1}}\sum_{\underset{1\leq m_1,m_2\leq m}{m\leq m_1+m_2\leq m+1}}U_B(k_1,m_1)U_B(k_2,m_2)\nonumber\\
    &\leq Q_k+\sum_{k_1=3}^{k-1}\sum_{m_1=2}^m\sum_{m_2=m-m_1}^{m-m_1+1}U_B(k_1,m_1)\widetilde{U}_B(2)+\sum_{\underset{3\leq k_2\leq k_1\leq k-1}{{k\leq}k_1+k_2\leq k+1}}\sum_{\underset{1\leq m_1,m_2\leq m}{m\leq m_1+m_2\leq m+1}}U_B(k_1,m_1)U_B(k_2,m_2)\nonumber\\
    &=:Q_k+P_1+P_2.
\end{align*} 
{In light of \eqref{under_cond1}, we have $\frac{(Cm)^{k-3}/(k-1)^k}{(Cm)^{k-2}/k^{k+1}}\leq c $
whenever $k\leq \tau m+1$. Therefore, if $4\leq k_1\leq k-1$ and $k\leq \tau m+1$, then $k_1\leq\tau m\leq\tau(m+1)+1$, so that
\begin{align*}
    \frac{(C(m+1))^{k_1-3}/((k_1-2)(k_1-1)^{k_1})}{(C(m+1))^{k_1-2}/((k_1-1)k_1^{k_1+1})}\leq c\frac{k_1-1}{k_1-2}\leq \frac{3c}{2}.
\end{align*}
Thus,
\begin{align*}
   \sum_{k_1=3}^{k-1}\frac{(C(m+1))^{k_1-2}}{(k_1-1)(k_1^{k_1+1})}\leq \frac{(C(m+1))^{k-3}}{(k-2)(k-1)^k}\Big(1+\frac{3c}{2}+\Big(\frac{3c}{2}\Big)^2+...\Big)=\frac{1}{1-\frac{3c}{2}}\frac{(C(m+1))^{k-3}}{(k-2)(k-1)^k},
\end{align*}
and} using the induction hypothesis, 
we get
\begin{align*}
    P_1&=2\widetilde{U}_B(2)\sum_{k_1=3}^{k-1}\sum_{m_1=2}^mU_B(k_1,m_1)\leq {10\sum_{k_1=3}^{k-1}\sum_{m_1=2}^m\ll\frac{C^{k_1-2}(m+1)^{k_1-2}}{k_1^{k_1+1}}}\\
    &\leq 10\sum_{k_1=3}^{k-1}\ll\frac{C^{k_1-2}(m+1)^{k_1-{1}}}{(k_1-1)k_1^{k_1+1}}    \leq 10\frac{1}{1-\frac{3c}{2}}\ll\frac{C^{k-3}(m+1)^{k-2}}{(k-2)(k-1)^{k}}\\
    &\leq 10\frac{1}{1-\frac{3c}{2}}\cdot \frac{k}{k-2}\cdot\Big(\frac{k}{k-1}\Big)^k\ll\frac{C^{k-3}(m+1)^{k-2}}{k^{k+1}}\\
    &\leq 10\frac{1}{1-\frac{3c}{2}}\cdot \frac{k}{k-2}\cdot\Big(\frac{k}{k-1}\Big)^ke^{\eta}\ll\frac{C^{k-3}m^{k-2}}{k^{k+1}},
\end{align*}
since $((m+1)/m)^{k-2}\leq e^{\eta}$ (recall \eqref{eta}). Noting that $(x/(x-2))(x/(x-1))^x\leq 2(x/(x-1))^x\leq  2\cdot(4/3)^4$ for $x\geq 4$, we conclude
\begin{align*}
       P_1\leq 10\frac{2^93^{-4}}{1-\frac{3c}{2}}e^{\eta}\ll\frac{C^{k-3}m^{k-2}}{k^{k+1}}\leq c_1\ll\frac{C^{k-2}m^{k-2}}{k^{k+1}}
\end{align*}
due to \eqref{cond0}.

For $P_2$, we have 
\begin{align*}
P_2&\leq \sum_{k_1=3}^{k-2}\sum_{k_2={\max}(3,k-k_1)}^{k-k_1+1}\sum_{m_1=1}^m\sum_{m_2=m-m_1}^{m-m_1+1}u(k_1,m_1)u(k_2,m_2)\\
&\leq 4\cdot \ll^2\sum_{k_1=3}^{k-2}\sum_{m_1=1}^m\frac{(Cm_1)^{k_1-2}}{k_1^{k_1+1}}\frac{(C(m-m_1+1))^{k-k_1-1}}{(k-k_1+1)^{k-k_1+2}}.
\end{align*}
{Note that the function $x\mapsto x^{k_1-2}(m+1-x)^{k-k_1-1}$ has two intervals of monotonicity in $[0,m+1]$, so that we can estimate}
\begin{align*}
    {\sum_{m_1=1}^m m_1^{k_1-2}(m-m_1+1)^{k-k_1-1}}&{<\int_0^{m+1}x^{k_1-2}(m+1-x)^{k-k_1-1}dx+\max_{x\in[0,m+1]}x^{k_1-2}(m+1-x)^{k-k_1-1}}\\
    &{=(m+1)^{k-2}B(k_1-1,k-k_1)}\\
    &{+\Big(\frac{m+1}{k-3}\Big)^{k-3}(k_1-2)^{k_1-2}(k-k_1-1)^{k-k_1-1}.}
\end{align*}
{According to this inequality, we can write $P_2<P_{21}+P_{22}$ and}
using that $B(k_1-1,k-k_1)=\Gamma(k_1-1)\Gamma(k-k_1)/\Gamma (k-1)$, 
 that for any positive integer $l$, $(l/e)^l\sqrt{2\pi l}\leq l!\leq (l/e)^l\sqrt{2\pi l}e^{1/12}$, {and that $(x-2)/(k-2)< x/k$ for all $x\in(0,k)$,} further obtain
\begin{align}\label{P_21}
P_{21}&\leq 4\cdot \ll^2 C^{k-3}(m+1)^{k-2}\sum_{k_1=3}^{k-2}\frac{\Gamma (k_1-1)\Gamma(k-k_1)}{\Gamma(k-1)}\frac{1}{k_1^{k_1+1}(k-k_1+1)^{k-k_1+2}}\nonumber\\
&\leq 4\cdot \ll^2e^{{7}/6}\sqrt{2\pi}C^{k-3}(m+1)^{k-2}\sum_{k_1=3}^{k-2}\frac{(k_1-2)^{k_1-3/2}(k-k_1-1)^{k-k_1-1/2}}{(k-2)^{k-3/2}k_1^{k_1+1}(k-k_1+1)^{k-k_1+2}}\nonumber\\
&\leq 4{e^{\eta}}\sqrt{2}\cdot \ll^2e^{{7}/6}\sqrt{2\pi}C^{k-3}m^{k-2}\frac{1}{k^{k-3/2}}\sum_{k_1=3}^{k-2}\frac{1}{k_1^{5/2}(k-k_1+1)^{5/2}},
\end{align}
as $((m+1)/m)^{k-2}\leq e^{\eta}$ and {$(k-2)^{-1/2}\leq \sqrt{2}k^{-1/2}$} for $k\geq 4$. 

{For $P_{22}$, we write}
\begin{align}\label{P_22}
{P_{22}}&{\leq {4} \ll^2 C^{k-3}\frac{(m+1)^{k-3}}{(k-3)^{k-3}}\sum_{k_1=3}^{k-2}\frac{(k_1-2)^{k_1-2}(k-k_1-1)^{k-k_1-1}}{k_1^{k_1+1}(k-k_1+1)^{k-k_1+2}}}\nonumber\\
&{<{4} \ll^2 C^{k-3}\frac{(m+1)^{k-3}}{(k-1)^{k-3}}\sum_{k_1=3}^{k-2}\frac{1}{k_1^{3}(k-k_1+1)^{3}}}\nonumber\\
&{\leq {4} \ll^2 C^{k-3}e^{\eta+1}\cdot\frac{1}{3}\cdot\frac{m^{k-2}}{k^{k-2}}\sum_{k_1=3}^{k-2}\frac{1}{k_1^{3}(k-k_1+1)^{3}}}
\end{align}
{in light of \eqref{eta} and the inequalities $(k/(k-1))^{k-3}\leq e$ and $((m+1)/m)^{k-2}\leq e^{\eta}$.}

To terminate the estimate for $P_1$, we will need the following

\begin{lemma}\label{lem_new}
    {We have}
    \begin{align*}
        {k^{\alpha}\sum_{k_1=3}^{k-2}\frac{1}{k_1^{\alpha}(k-k_1+1)^{\alpha}}<}\begin{cases}
            {0.71,}&{ \alpha=5/2,}\\
            {1.{5},}&{\alpha =3.}
        \end{cases}
    \end{align*}
\end{lemma}

\begin{proof}
{Since the function $x^{-\alpha}(k-x+1)^{-\alpha}$ is decreasing on $(0,(k+1)/2)$ and increasing on $(k+1)/2, k+1)$,  
 the sum we need to bound is less than}
\begin{align*}
    {F_{\alpha}(k):=k^{\alpha}\int_{2}^{k-1}x^{-\alpha}(k-x+1)^{-\alpha}dx.} 
\end{align*}
Making two changes of variable $y:=x/(k+1)$ and $z:=y/(1-y)$, we derive
\begin{align*}
    {F_{\alpha}(k)=k^{\alpha}(k+1)^{-2\alpha+1}\int_{\frac{2}{k+1}}^{\frac{k-1}{k+1}}y^{-\alpha}(1-y)^{-\alpha}dy=k^{\alpha}(k+1)^{-2\alpha+1}\int_{\frac{2}{k-1}}^{\frac{k-1}{2}}z^{-\alpha}(z+1)^{2\alpha-2}dz,}
    \end{align*}
    so that
    \begin{align*}
        F_{5/2}(k)&=k^{5/2}(k+1)^{-4}{\int_{\frac{2}{k-1}}^{\frac{k-1}{2}}}(z^{1/2}+3z^{-1/2}+3z^{-3/2}+z^{-5/2})dz\\ 
    &=2k^{5/2}(k+1)^{-4}\Big(\frac{(k-1)^{3/2}}{3\sqrt{2}}+3\sqrt{2}(k-1)^{1/2}-{6}\sqrt{2}(k-1)^{-1/2}-\frac{{4\sqrt{2}}(k-1)^{-3/2}}{3}\Big)\\
    &=2(s^2+1)^{5/2}(s^2+2)^{-4}\Big(\frac{s^{3}}{3\sqrt{2}}+3\sqrt{2}s-{6}\sqrt{2}s^{-1}-\frac{{4\sqrt{2}}s^{-3}}{3}\Big)=:\widetilde{F}_{5/2}(s),
\end{align*}
where $s=s(k)=\sqrt{k-1}$. The only point {$x>0$} of $\widetilde{F}_{5/2}'(x)=0$ is $x=x_0\approx 3.18{2}$, so that verifying 
{$F_{5/2}(x)<0.71$} for $k=4$ and $k=x_0^2+1\approx 11.126$ concludes the proof of {the lemma for $\alpha=5/2$}.

{For $\alpha=3$, {using that $k\geq 5$}, we get}
\begin{align*}
    {F_{3}(k)=k^{3}(k+1)^{-5}(0.5z^{2}+4z{-\frac{4}{z}-\frac{z^2}{2}}+6\log z)\Big|_{\frac{2}{k-1}}^{\frac{k-1}{2}}<{2}k^{-2}\Big(\frac{k^2}{8}+2k+{6}\log\frac{k}{2}\Big)<{1.5}.}
\end{align*}
\end{proof}

Estimating {the sums in \eqref{P_21} and \eqref{P_22} according to Lemma \ref{lem_new},} 
 we arrive at
\begin{align*}
    P_2&{< P_{12}+P_{22}\leq e^{\eta}\lambda\Big(\sqrt{2}\cdot 4\cdot 0.71 e^{{7}/6}\sqrt{2\pi}+\frac{{4}e\cdot{1.5}}{3}\Big)\cdot \ll\frac{C^{k-3}m^{k-2}}{k^{k+1}}}\leq c_2\cdot \ll\frac{(Cm)^{k-2}}{k^{k+1}}
\end{align*}
under condition \eqref{cond2}.

Finally,
\begin{align*}
    N_2\leq Q_k+(c_1+c_2)\cdot\ll\frac{(Cm)^{k-2}}{k^{k+1}}.
\end{align*}

\textbf{Case 3.} Note that $v\in \{1\}\cup B\cap D_1\setminus\{3\}=\{2^s,\;0\leq s\leq 4\}=:V$ ($v\neq 3$, since otherwise $n=(a+b)v\in B$ with $\|n\|=\|3\|+\|n/3\|$, which is a contradiction).

Besides, $nv^{-1}\in(3^{m(v)-2},3^{m(v)}]$, where $m(v)=m-\lfloor \log_3 v\rfloor$, and if $a\in(3^{m_1-1},3^{m_1}]$ and $b\in(3^{m_2-1},3^{m_2}]$, then
\begin{align*}
    \de(nv^{-1})-\de(a)-\de(b)&=\|nv^{-1}\|-\|a\|-\|b\|+3\log_3 \frac{abv}{n}\\
    &=3\log_3 \frac{abv}{n}\geq 3(m_1+m_2-m(v)-2),
\end{align*}
{so that
\begin{align*}
    k\sigma>\de(n)\geq\de(nv^{-1})\geq \de(a)+\de(b)+ 3(m_1+m_2-m(v)-2)
\end{align*}
and thus, as $a\geq nv^{-1}/2$ and therefore $m_1\geq m(v)-1$,
\begin{align}\label{k_sig}
    \de(a)&\leq k\sigma-3(m_1+m_2-m(v)-2)\nonumber\\
    &<(k-3(m_1+m_2-m(v)-2))\sigma\leq (k-3(m_2-3))\sigma.
\end{align}
}
We divide this case into two separate ones corresponding to $b\leq 27$ and $b>27$ and denote the numbers of such representations, respectively, as $N_{31}$ and $N_{32}$. We emphasize that here we are not going to use the bound on $\|b\|$ given above in the description of decomposition (3). For the representations corresponding to $b\leq 27$, taking into account that $|V|=5$ and that $b$ must be additively irreducible (whence, by a direct check, $b\in Z:=\{1,6,8,9,12,14,15,16,18,20,21,24,26,27\}$ with $|Z|= 14$), we write 

\begin{align*}
    N_{31}&\leq \sum_{b\in Z}\sum_{v\in V}\sum_{p=1}^{k-1}|D_p\cap(3^{m-1}v^{-1}-b,3^mv^{-1}-b]|=\sum_{b\in Z}\sum_{v\in V}\Big(\sum_{p=1}^2+\sum_{p=3}^{k-1}\Big)=:L_1+L_2.
    \end{align*}
    Observe that if $3\mid b$, {then} 
     $a$ must be a leader. {Indeed, it can be seen that any number $3\mid x\leq 27$ can be efficiently represented as $x=3\cdot (x/3)$, whence if $a$ is not a leader, then the number $a+b=3(a/3+b/3)$ has complexity less than $\|a\|+\|b\|$, which is a contradiction}. Moreover, one  can check that the maximal number of elements in $B\cap (D_1\cup D_2)$ contained in an interval of the form $(x,4x]$ is $7$ ($\{5,7,8,10,14,16,19\}\subset (4.75,19]$) and in light of \eqref{eta}, $m\geq 12$, so that $4(3^{m-1}v^{-1}-b)>3^mv^{-1}-b$. Thus, noting that $|Z\cap 3\mathbb Z|=8$, we obtain
\begin{align}\label{L_1}
    L_1\leq \sum_{v\in V}\Big(\sum_{3\mid  b}7+\sum_{3\nmid  b} {2\cdot} 17\Big)=5\cdot(8\cdot 7+6\cdot {2\cdot} 17)={1300}.
\end{align}
Also, using that $U_{\mathbb N}(s,l)\leq \sum_{r\leq l}U_B(s,r)$ {(since any number $n$ has a unique leader $n_0$ asigned to it with $n/n_0$ being a power of $3$) and that $a\in (3^{m(v)-3},3^{m(v)}]$ (since $a\leq nv^{-1}$ and $a\geq (a+b)/2=nv^{-1}/2$, with $nv^{-1}\in (3^{m(v)-2},3^{m(v)}]$), we have} 
    \begin{align*}
    L_2&\leq \sum_{v\in V}\sum_{p=3}^{k-1}\Big(\sum_{3\mid b\in Z}\sum_{m_1=m(v)-2}^{m(v)} U_{B}(p,m_1)+\sum_{3\nmid b\in Z}\sum_{m_1=1}^m U_{B}(p,m_1) \Big)\\
    &\leq 5\sum_{p=3}^{\min(k-1,\tau m+1)}\Big(8\cdot 3\ll\frac{(Cm)^{p-2}}{p^{p+1}}+6\ll\frac{C^{p-2}(m+1)^{p-1}}{(p-1)p^{p+1}}\Big)\\
    &\leq \frac{120}{1-c}\ll\frac{(Cm)^{k-3}}{(k-1)^k}+\frac{30e^{\eta}}{1-\frac{3c}{2}}\cdot \ll\frac{C^{k-3}m^{k-2}}{(k-2)(k-1)^{k}}\\
     &\leq \Big(30e^{\eta}+\frac{120}{6}\Big)\frac{1}{1-\frac{3c}{2}}\cdot \ll\frac{C^{k-3}m^{k-2}}{(k-2)(k-1)^{k}}\\
     &< \frac{{55.5}}{1-\frac{3c}{2}}\cdot 2\cdot \Big(\frac{4}{3}\Big)^4\cdot \ll\frac{C^{k-3}m^{k-2}}{k^{k+1}}\leq c_3\cdot \ll\frac{(Cm)^{k-2}}{k^{k+1}} 
\end{align*}
in light of condition \eqref{cond4} and the inequality $((m+1)/m)^{k-2}\leq e^{\eta}$ (recall \eqref{eta}), where in the third to last line we used \eqref{under_cond1}. Thus,
\begin{align*}
    N_{31}\leq L_1+L_2\leq {1300}+c_3\cdot \ll\frac{(Cm)^{k-2}}{k^{k+1}}.
\end{align*}

Turn now to the case of $b>27$, i.e. to estimating $N_{32}$. {We will divide this case into three subcases: those of $\de(a),\de(b)<2\sigma,\;\de(a)<2\sigma\leq \de(b)$, and $\de(a)\geq 2\sigma$. Observe that for any $s$, there are at most $|B\cap (D_1\cup D_2)|=17$ numbers in $(3^s,3^{s+1}]$ with the defect less than $2\sigma$.}
 Taking into account this observation {along with \eqref{k_sig} and the facts that $a\in (3^{m_1-1},3^{m_1}]\subset (3^{m(v)-3},3^{m(v)}]$ and $b\in(3^{m_2-1},3^{m_2}]$, $m_2\leq m_1\leq m$,} we get 
\begin{align*}
    N_{32}&\leq {\sum_{v\in V}\sum_{m_2=4}^{m}3\cdot 17\cdot 17}+\sum_{v\in V}\sum_{m_2=4}^{m}\sum_{p=3}^{k-3(m_2-3)}51U_{\mathbb N}(p,m_2)\\
    &+\sum_{v\in V}\sum_{m_1=m(v)-2}^{m(v)}\sum_{m_2=4}^{m_1}\sum_{p=3}^{k-3(m_1+m_2-m(v)-2)}U_{\mathbb N}(p,m_1)3^{m_2}\\
    &=: {4335(m-3)}+M_1+M_2,
\end{align*}
{where we used that the number of $b$'s in $(3^{m_2-1},3^{m_2}]$ is less than $3^{m_2}$}. Also, the inequality 
 $k-3(m_1+m_2-m(v)-2)\leq k-3(m_2-3)\leq k-3$ implies that if $k<{6}$, then ${M_1=M_2}=0$, whence from now on we assume $k\geq {6}$ for the estimates of $M_1$ and $M_2$. 

For $M_1$, in light of \eqref{under_cond1}, we obtain
\begin{align*}
    M_1 &\leq 255\sum_{m_2=4}^m \sum_{p=3}^{\min(k-3(m_2-3),\tau m_2+1)}\ll\frac{C^{p-2}m_2^{p-1}}{(p-1)p^{p+1}}\\
  &\leq 382.5\sum_{m_2=4}^m \sum_{p=3}^{\min(k-3(m_2-3),\tau m_2+1)}\ll\frac{C^{p-2}m_2^{p-1}}{p^{p+2}}\\
  &= 382.5\sum_{p=3}^{k-3}\sum_{m_2={\max(4},\lceil (p-1)/\tau\rceil)}^{\min(m,(k+9-p)/3)}\ll\frac{C^{p-2}m_2^{p-1}}{p^{p+2}}\\
  &\leq  382.5\sum_{p=3}^{k-3}\ll\frac{C^{p-2}(m+1)^{p}}{p^{p+3}}\leq 382.5e^{\eta}\cdot\frac{c}{1-c}\cdot \ll\frac{C^{k-4}m^{k-2}}{(k-2)^{k+1}}\\
  &\leq \frac{382.5e^{\eta}c}{1-c}\cdot{\Big(\frac{6}{4}\Big)^7}\cdot \ll\frac{C^{k-4}m^{k-2}}{k^{k+1}}\leq c_4\cdot \ll\frac{(Cm)^{k-2}}{k^{k+1}},
\end{align*}
{since $(k/(k-2))^{k+1}$ is decreasing} and due to \eqref{cond5}.

Next,
\begin{align*}
    M_2&\leq 15\sum_{m_2=4}^{\min(m,(k+{6})/3)}3^{m_2}\sum_{p=3}^{k-3(m_2-3)}\sum_{m_3=1}^m\ll\frac{(Cm_3)^{p-2}}{p^{p+1}}\\
    &\leq 15\sum_{m_2=4}^{\min(m,(k+{6})/3)}3^{m_2}\sum_{p=3}^{k-3(m_2-3)}\ll\frac{C^{p-2}(m+1)^{p-1}}{(p-1)p^{p+1}}\\
      &\leq 15e^{\eta}\sum_{m_2=4}^{\min(m,(k+{6})/3)}3^{m_2}\sum_{p=3}^{k-3(m_2-3)}\ll\frac{C^{p-2}m^{p-1}}{(p-1)p^{p+1}}\\
     &\leq \frac{22.5e^{\eta}}{1-\frac{3c}{2}}\cdot \ll\sum_{m_2=4}^{\min(m,(k+{6})/3)}3^{m_2}\frac{C^{k-3(m_2-3)-2}m^{k-3(m_2-3)-1}}{(k-3(m_2-3))^{k-3(m_2-3)+2}}\\
     &=:\frac{22.5e^{\eta}}{1-\frac{3c}{2}}\cdot \ll\sum_{m_2=4}^{\min(m,(k+{6})/3)}G(k,m,m_2).
\end{align*}
{Note that $$\Big((x+2)\log\frac{x+3}{x}\Big)'=\log\frac{x+3}{x}-\frac{3(x+2)}{x(x+3)}<\frac{3}{2x}-\frac{3(x+2)}{x(x+3)}<0$$ for $x\geq 3$, whence the function $((x+3)/x)^{x+2}$ is decreasing for $x\geq 3$. Thus, if $4\leq l\leq (k+6)/3-1$, so that $k-3(l-2)\geq 3$, in light of \eqref{cond1},} we have
\begin{align*}
\frac{G(k,m,l)}{G(k,m,l+1)}> &{\frac{C^3m^3}{3(k-3(l-3))^3}\Big(\frac{k-3(l-2)}{k-3(l-2)+3}\Big)^{k-3(l-2)+2}}\\
&{\geq \frac{C^3m^3}{2^5 3(k-3)^3}>\frac{(2^{10}3^{-5}c^{-1})^3}{96}>\frac{1}{1.3c^3}.}
\end{align*} 
whence, since {$c<0.9<1.3^{-1/3}$}
\begin{align*}
    M_2&\leq \frac{1}{1-420c^3}\cdot\frac{22.5e^{\eta}}{1-\frac{3c}{2}}\cdot 3^{4}\cdot \ll\frac{C^{k-5}m^{k-4}}{(k-3)^{k-1}}\\
    & \leq \Big(\frac{5}{4}\Big)^5c^2\frac{1}{1-420c^3}\cdot\frac{22.5e^{\eta}}{1-\frac{3c}{2}}\cdot 3^{4}\cdot \ll\frac{C^{k-3}m^{k-2}}{k^{k+1}}\leq c_5\cdot \ll\frac{(Cm)^{k-2}}{k^{k+1}}
\end{align*}
according to \eqref{under_cond1} and \eqref{cond33}. Hence, $N_{32}\leq {4335(m-3)}+M_1+M_2\leq (c_4+c_5)\cdot \ll(Cm)^{k-2}/k^{k+1}$.

Note that due to \eqref{eta}) we can assume $m\geq 12$ for $k\geq 4$, so that 
$$c_6\ll \frac{(Cm)^{k-2}}{k^{k+1}}\geq c_6\ll \frac{(Cm)^{2}}{4^5}{\geq \frac{1309m^2}{12^2}+\frac{4335m^2\cdot 9}{12^2}\geq 1309 +4335(m-3)}\geq Q_k+{1300}+{4335(m-3)}$$
according to \eqref{cond8} and the fact that $Q_k\leq 9$. Taking this into account and combining all the estimates above, we arrive at
\begin{align*}
U_B(k,m)&\leq N_1+N_2+N_{31}+N_{32}\\
    &\leq Q_k+{1300}+{4335(m-3)}+\sum_{j=1}^5c_j\cdot \ll\frac{(Cm)^{k-2}}{k^{k+1}}\\
    & \leq \sum_{j=1}^6c_j\cdot \ll\frac{(Cm)^{k-2}}{k^{k+1}}= \ll\frac{(Cm)^{k-2}}{k^{k+1}},
\end{align*}
which completes the induction step.

Turn now to the base of induction. {Note that $3^m<2.5(780m)/3^4$ whenever $m\leq 4$, so that we can assume now that $m\geq 5$.} Let us use once again decompositions (1)-(3) that led us to Cases 1-3 above. {Note that the only number obtained in Case 1 is $80=5\cdot 16,$ where $5,16\in B\cap (D_1\cup D_2)$, so we will take this number into account in $Q_3$ below when dealing with Case 2}. Using that in decomposition (3) for our case we have $\|b\|<3\cdot 0.48+3\log_32<4$, we conclude that {$b=1$, since $2$ and $3$ are not additively irreducible.} 
Recalling also the argument before \eqref{L_1}, which remains true for $m= 11$, {so that, given $v\in V$ and $x\in B\cap (D_1\cup D_2)$, there are at most two values for $a$ of the form $3^tx$}, we derive 
\begin{align*}
    U_B(3,m)&\leq Q_3+\sum_{v\in V}\Big(|{(}D_1\cup D_2{)}\cap (3^{m-1}v^{-1}-1,3^m v^{-1}-1]|\\
    &\leq 8+ {5\cdot 2\cdot |B\cap(D_1\cup D_2)|=8+5\cdot 2\cdot 17=178}\leq \ll\frac{Cm}{3^4}
\end{align*}
 due to \eqref{00} for $m\geq 11$. For $m=5,...,10$, we check directly\footnote{Access the code $``$defect\_check\_local.py$"$ via the \Colorhref{https://drive.google.com/drive/folders/1Dse6fijPlJLgI_mJROqsm3zeMFVFC20q?usp=drive_link}{\underline{link}}.} that $U_{{\mathbb{N}}}(3,m)=36,55,73,89,105,120$, respectively, so that for all of these values of $m$, we have $U_{{B}}(3,m)/m< {178}/11$, which completes the proof of the induction base.

Thus, \eqref{we_want} is proved. 

Finally, taking $\gamma>0$ that satisfies $\gamma\sigma^{-1}\log_3 C\sigma\gamma^{-1}<1$ and using \eqref{we_want}, we get  
\begin{align*}
|\{n\leq x:{2\sigma\leq} \de(n)\leq\gamma \log_3 n\}|&\leq \sum_{m=1}^{\log_3 x+1}{\sum_{k=3}^{\lfloor\frac{\gamma m}{\sigma}\rfloor+1}}U_{\mathbb N}(k,m){\leq \sum_{m=1}^{\log_3 x+1}{\sum_{k=3}^{\lfloor\frac{\gamma m}{\sigma}\rfloor+1}}\sum_{l=1}^m U_B(k,l)}\\
& {\leq \sum_{m=1}^{\log_3 x+1}{\sum_{k=3}^{\lfloor\frac{\gamma m}{\sigma}\rfloor+1}}\lambda\sum_{l=1}^m \frac{(Cl)^{k-2}}{k^{k+1}}}\leq 
 \ll\sum_{m=1}^{\log_3 x+1}m{\sum_{k=3}^{\lfloor\frac{\gamma m}{\sigma}\rfloor+1}}\frac{(Cm)^{k-2}}{k^{k+1}}\\
 &\lesssim\log x \sum_{m=1}^{\log_3 x+1}{\frac{(Cm)^{\lfloor\frac{\gamma m}{\sigma}\rfloor-1}}{(\lfloor\frac{\gamma m}{\sigma}\rfloor+1)^{\lfloor\frac{\gamma m}{\sigma}\rfloor+2}}}\lesssim \log x\sum_{m=1}^{\log_3 x+1}\Big(\frac{C\sigma}{\gamma}\Big)^{\frac{\gamma m}{\sigma}-1}\\
 &\lesssim \log x\Big(\frac{C\sigma}{\gamma}\Big)^{\frac{\gamma }{\sigma}\log_3 x}=x^{\frac{\gamma }{\sigma}\log_3\frac{C\sigma}{\gamma}}\log x=o(x),
\end{align*} 
where we also used that 
 $C\sigma/\gamma>1$. Hence, {since $|\{n\leq x:\de(n)<2\sigma\}|=O(\log x)$,} for almost all numbers $n$, we have $\|n\|\geq (3+\gamma)\log_3 n$. Now, in order to prove Theorem \ref{th_low}, it suffices to show that there is an appropriate choice of 
 $C$ for $\gamma=0.06$.

If we choose $C:=780$, {$\lambda:=2.5$,}  
 so that 
  \eqref{00} {holds}, and put $c:={2^{10}3^{-5}}\tau/ 780<{0.9}$ (note that for such choice we have the condition $c^3\leq 420^{-1}$ fulfilled), 
 we can see that the sum of the lower bounds \eqref{cond0}-\eqref{cond8} for $c_j,\;1\leq j\leq 6,$ is
\begin{align*}
    \frac{127e^{\eta}}{2-3c}C^{-1}&+ {38}e^{\eta}\ll C^{-1}+\frac{{702}}{2-3c} C^{-1}+\frac{{6536}e^{\eta}c}{1-c}C^{-2}\\
    &+\frac{11124c^2e^{\eta}}{(1-{1.3}c^3)(2-3c)} C^{-1}+{286749}\ll^{-1} C^{-2}<1,    
\end{align*}
whence we can find appropriate values for $c_j,\;1\leq j\leq 6,$ to satisfy all the necessary conditions.

 With such choice we have that $\gamma:=0.06$ satisfies $\gamma\sigma^{-1}\log_3 C\sigma\gamma^{-1}<1$, which finally completes the proof of Theorem \ref{th_low}.
  \end{proof}

 \begin{remark}
 One can see that for other $\sigma<0.48$, inequality \eqref{we_want} with some $\lambda=\lambda(\sigma)$ and $C=C(\sigma)$ can be proved in the same fashion.
 \end{remark}

\begin{acknowledgements} {We are very grateful to Juan Arias de Reyna for carefully reading the paper and pointing out several inaccuracies in the first version of the manuscript as well as to Harry Altman, Alexander Kalmynin, and Qizheng He for providing us with useful references.} We also thank Artem Badalyan for his help with numerical experiments.
\end{acknowledgements}


\begin{thebibliography}{65}

\bibitem{A}
H. Altman, {\it Integer complexity: representing numbers of bounded defect}, Theor. Comput. Sci. 652 (2016), 64--85.

\bibitem{A2}
{H. Altman, {\it Integer complexity: algorithms and computational results}, Integers 18(45)  (2018).}

\bibitem{AZ}
H. Altman, J. Zelinsky, {\it Numbers with integer complexity close to the lower bound}, Integers 12(6)  (2012), 1093--1125.


\bibitem{Am}
{K. Amano, {\it Integer complexity and mixed binary-ternary representation}, Leibniz Int. Proc. Inform.  
(LIPIcs) 248(29)  (2022), 1--16.}

\bibitem{Ar}
J. Arias de Reyna, {\it Complexity of integer numbers}, Integers 24  (2024) (translation from {\it Complejidad de los n\'umeros naturales}, Gaceta R. Soc. Mat. Esp. 3  (2000), 230--250).

\bibitem{AL}
J. Arias de Reyna, J. van de Lune, {\it Algorithms for determining integer complexity}, arXiv:1404.2183.

\bibitem{CEHMPSSV}
K. Cordwell, A. Epstein, A. Hemmady, S. J. Miller, E. Palsson, A. Sharma, S. Steinerberger, Y. N. T. Vu, {\it On algorithms to calculate integer complexity}, Integers 19  (2019), A12.

\bibitem{G1}
R.K. Guy, {\it Some suspiciously simple sequences}, Amer. Math. Monthly 93(3)  (1986), 186--190. 

\bibitem{H}
{Q. He, {\it Improved algorithms for integer complexity}, Proceedings of 2024 Symposium on Simplicity in Algorithms (SOSA), 107--114.}

\bibitem{IBCOOP}
J. Iraids, K. Balodis, J. \u{C}er\c{n}enoks, M. Opmanis, R. Opmanis, K. Podnieks, {\it Integer complexity: experimental ana analytical results}, arXiv:1203.6462, Scientific Papers University of Latvia, Comp. Sci. and Inf. Tech. 787 (2012), 153--179.

\bibitem{IK}
H. Iwaniec, E. Kowalski, {\it Analytic number theory}, AMS, Providence, 2004.

\bibitem{KN}
L. Kuipers, H. Niederreiter, {\it Uniform distribution of sequences}, Wiley, New York, 1974.

\bibitem{MP}
K. Mahler, J. Popken, {\it On a maximum problem in arithmetics (Dutch)}, Nieuw Arch. Wiskunde 3(1)  (1953), 1--15.

\bibitem{R}
D.A. Rawsthorne, {\it How many $1$'s are needed?}, Fibonacci Quart. 27(1)  (1989), 14--17.

\bibitem{Sh}
{C. Shriver, {\it An application of Markov chain analysis to integer complexity}, arXiv:1511.07842.}

\bibitem{Sl}
N.J.A. Sloane, {\it The on-line encyclopedia of integer sequences. The complexity of $n$: number of $1$'s required to build $n$ using $+$ and $\cdot$}, https://oeis.org/A005245.

\bibitem{S}
S. Steinerberger, {\it A short note on integer complexity}, Contrib. Discrete Math. 9(1) (2014).

\bibitem{Z}
J. Zelinsky, {\it Upper bounds on integer complexity}, arXiv:2211.02995.



\end{thebibliography}
\end{document}